\documentclass[12pt]{amsart}

\title{Terminal orders on arithmetic surfaces}

\author{Daniel Chan}
\thanks{DC was supported by the Australian Research Council Discovery Project grant  DP220102861}
\address{School of Mathematics and Statistics, 
UNSW Sydney, 
NSW 2052,
Australia
}
\email{danielc@unsw.edu.au}

\author{Colin Ingalls}
\thanks{CI was partially supported by a Discovery Grant from the
  National Science and Engineering Research Council of Canada}
\address{School of Mathematics and Statistics, Carleton University, Ottawa, ON K1S 5B6, Canada}
\email{cingalls@math.carleton.ca}

\usepackage{mymacros,amssymb}
\usepackage{amsmath}
\usepackage{amsfonts}
\usepackage{amscd}
\usepackage{hyperref}

\thanks{}

\begin{document}
\maketitle

\begin{abstract}
The local structure of terminal Brauer classes on arithmetic surfaces were classified in \cite{CI24} generalising the classification on geometric surfaces carried out in \cite{CI05}. Part of the interest in these classifications is that it enables the minimal model program to be applied to the noncommutative setting of orders on surfaces. In this paper, we give \'etale local structure theorems for terminal orders on arithemtic surfaces, at least when the degree is a prime $p >5$. This generalises the structure theorem given in the geometric case. They can all be explicitly constructed as algebras of matrices over symbols. From this description one sees that such terminal orders all have global dimension two, thus generalising the fact that terminal (commutative) surfaces are smooth and hence homologically regular. 
\end{abstract}

\section{Introduction}  \label{sec:intro}
Given a smooth point on a complex variety of dimension $d$, the \'etale local structure is $\Spec R$ where $R =  \bC\{x_1,\ldots, x_d\}$, the algebra of algebraic power series in $d$ variables. This result elegantly captures, in an algebraic fashion, the idea that manifolds are all locally Euclidean. In \cite{CI05}, a noncommutative analogue of this result was given for the case of orders on a complex surface. Here we extend the result to arbitrary surfaces, at least under some mild hypotheses. 

More precisely, let $R$ be an excellent noetherian normal two dimensional Hensel local domain with fraction field $K$. Let $\Lambda$ be a maximal $R$-order in some central simple $K$-algebra. It is natural to associate the Brauer class $\beta_{\Lambda} \in \textup{Br}\, K,$ to this algebra, so one can use the minimal model program for surfaces enriched by a Brauer class as developed in \cite{CI05} and \cite{CI24}. More generally, if $\Lambda$ is a normal $R$-order (see Definition~\ref{def:normalorder}) then we attach a localised Brauer class $(\beta,g_{\frakp})$ (in Definition~\ref{def:localisedBrauer}) where the integers $g_{\frakp}$ measure how much $\Lambda$ deviates from being maximal at the codimension one prime $\frakp$. 

When the residue field $\kappa$ is algebraically closed, terminal localised Brauer classes $(\beta,g_{\frakp})$, and the corresponding terminal orders were completely classified in \cite{CI05} and shown to always have global dimension two. In \cite{CI05}, it is shown that i) $R$ is smooth, ii) $\beta$ is zero or has ramification along a normal crossing divisor $C_1 \cup C_2$, and iii) $\Lambda$ is maximal except possibly along a curve of multiplicity one when $\beta=0,$ or one of the $C_i$ otherwise. Furthermore, a complete structure theorem was given for terminal normal orders in this case (see \cite[Section~2]{CI05}). If $\Lambda$ is maximal, then it is isomorphic to a full matrix algebra over a symbol $\Delta = R\langle y,z \rangle / (y^m - u, z^m - v, zy - \zeta yz)$ where $\zeta$ is some primitive $m$-th root of unity and $u,v \in R$ is a regular system of parameters. In general, one obtains a triangular modulo $z$ matrix algebra over such symbols as defined in  Definition~\ref{def:triangularMatrix}.

When $\kappa$ is a finite field of characteristic prime to the order $m$ of $\beta$, it turns out there are more possibilities for terminal localised Brauer classes. If $m$ is a prime $>5$, the terminal localised Brauer classes were completely classified in \cite{CI24}. The new possibilities are summed up in Definitions~\ref{def:stableno2ndram} and \ref{def:terminalHJ}, but briefly, when $R$ is regular, there is the additional possibility that $\beta$ is ramified on a single multiplicity one curve, and more interestingly, $R$ can also be a type of Hirzebruch-Jung singuluarity, in which case $\beta \neq 0,$ though it is unramified along codimension one primes of $R$. Our main theorem is the following result which generalises the aforementioned structure theory of terminal normal orders to this arithmetic situation. Again, symbols feature significantly but in the more general sense of tensor products of a $\mathbb{Z}/m$-extension of $R$ with a $\mu_m$-extension and graded components skew commute according to the natural pairing $\mathbb{Z}/m \times \mu_m \to \mu_m$ (see  Definition~\ref{def:symbol}). 

\begin{theorem}  \label{thm:intromain}
Let $R$ be an excellent noetherian two-dimensional normal Hensel local domain. Let $\Lambda$ be a terminal  $R$-order whose degree is a prime, say $m>5$. If the residue field of $R$ contains a primitive $m$-th root of unity and has trivial Brauer group, then $\Lambda$ has global dimension two. In fact, all such $\Lambda$ are explicitly constructed in Propositions~\ref{prop:structurestable2ndram}, \ref{prop:structuretoralterminalregno2nd} and Theorem~\ref{thm:mainHJ} as various triangular modulo $z$ matrix algebras  over symbols.
\end{theorem}

To prove the theorem, we construct explicit examples of normal orders with all the possible terminal localised Brauer classes. In the case when $R$ is regular, this is relatively straight forward. In the singular case, we need to show that our Hirzebruch-Jung singularity, defined to have a minimal resolution whose exceptional locus is a string of projective lines defined over the residue field of $R$, is actually a cyclic quotient singularity. We prove this in Section~\ref{sec:HJcyclic}, and give an explicit construction of the regular cyclic cover which is used in the construction of the corresponding terminal orders. The other step is to show that these explicitly constructed orders are sufficiently nice (in particular, have global dimension two) and that any normal order with the same ramification data has to be Morita equivalent to them. We follow the basic framework of \cite{CI05}, Unfortunately, the use of the Cohen-Structure theorem in \cite{CI05} is unavailable in this setting, so we give a streamlined method avoiding this tool in Section~\ref{sec:almostmaximal}.

\section{Uniqueness result for regular almost maximal orders}  \label{sec:almostmaximal}

In this section, we review basic definitions of orders, their ramification theory and the relationship with the Brauer group. Since the notion of maximal orders is not stable under \'etale localisation, we review normal orders as introduced in \cite{CI05}. The main result is a uniqueness type result for certain normal orders which have global dimension two and are maximal everywhere except possibly on a single irreducible divisor. This was proved in \cite{CI05} over an algebraically closed field using some complicated argument in \S 2.3. We present a streamlined proof here using the classification of normal orders over discrete valuation rings found in the Appendix \S \ref{sec:appendix}.

Let $R$ be a noetherian normal domain and $K$ its field of fractions. Given a central simple $K$-algebra $Q$, an {\em order} $A$ in $Q$ is an $R$-subalgebra such that $A$ is a finitely generated $R$-module such that $KA=Q$. Then in fact $K \otimes_R A \simeq Q$ so we sometimes dispense with explicitly mentioning $Q$ and say a finite $R$-algebra $A$ is an $R$-order if it is a torsion-free $R$-module such that $K \otimes_R A$ is a central simple $K$-algebra. We define the {\em degree} of $A$ to be the $\deg A := \deg K \otimes_R A  = \sqrt{\dim_K K \otimes_R A}$.

One ought to think of $A$ as a model of the noncommutative ``field'' $Q$ in this case. The classical noncommutative analogue of the notion of normality, is that the order is {\em maximal}, that is, if $A'$ is another order in $Q$ containing $A$, then $A = A'$. Unfortunately, this notion is not stable under \'etale localisation. When $R$ is two-dimensional, the following condition was introduced in \cite{CI05} to remedy this defect, taking its cue from Serre's criterion for normality in the commutative case.
\begin{definition} \label{def:normalorder}
Let $R$ be a two-dimensional normal domain. An $R$-order $A$ is said to be {\em normal} if the following two conditions hold.
\begin{enumerate}
    \item $A$ is a reflexive $R$-module.
    \item For every height one prime $\frakp$, the localisation $A_{\frakp}$ is {\em normal} in the sense that its radical is principal as a left and right ideal.
\end{enumerate}
\end{definition}
The second condition is thoroughly analysed in the Appendix \S~\ref{sec:appendix}. Note that maximal orders are normal, and that normal orders are tame. 

Let $R$ be a two-dimensional normal domain and $A$ be a normal $R$-order. Since $K \otimes_R A$ is a central simple $K$-algebra, it determines a corresponding Brauer class $\beta_A \in \text{Br}\, K$. Given any codimension one prime $\frakp \triangleleft R$,  with corresponding residue field $\kappa(\frakp)$, there is a ramification map 
$$
a_{\frakp} \colon \text{Br}\, K \longrightarrow H^1_{\acute{e}t}(\kappa(\frakp),\bQ/\bZ).
$$
As noted by Artin and Mumford~\cite{AM}, this map can be interpreted in terms of orders as follows. First note that $a_{\frakp}(\beta_A)$, being an element of $H^1_{\acute{e}t}(\kappa(\frakp),\bQ/\bZ)$ is given by a cyclic field extension $\kappa'$ of $\kappa(\frakp)$ and a choice of generator $\sigma$ for the Galois group $\text{Gal}(\kappa'/\kappa(\frakp))$. Let $J$ be the radical of $A_{\frakp}$ which is principal, and so is generated by an element, say $\pi$. Then $\kappa' = Z(A_{\frakp}/J)$ and $\sigma$ is the automorphism induced by conjugation by $\pi$. Note that $a_{\frakp}(\beta_A) = 0$ means that $A_{\frakp}$ is Azumaya so the collection of non-zero $a_{\frakp}(\beta_A)$ is called the ramification data of $A$. The above results are described in more detail in~\cite[\S 1]{ArtindeJong},~\cite{C11},\cite[\S 4]{CThesis},\cite[\S 1]{GI}.

If now $B$ is a normal order contained in the maximal order $A$ above, then from the Appendix, we know that $Z(B_{\frakp}/\text{rad}\, B_{\frakp}) \simeq \prod_{i=1}^d \kappa'$ for some $d$. Furthermore, conjugation by a generator $t$ of $\text{rad}\, B_{\frakp}$ permutes the $d$ factors cyclically, and conjugation by $t^d$ reduces to $\sigma$. The ramification data of $B$ will thus not only include the $a_{\frakp}(\beta_A)$, but also the integers $g_{\frakp} := d$. Since $B$ is generically Azumaya and thus maximal, we find on varying $\frakp$ that all but finitely many of the $g_{\frakp}$ will be one. 

\begin{definition}  \label{def:localisedBrauer}
A {\em localised Brauer class on $R$} is a pair $(\beta, g_{\frakp})$ consisting of a Brauer class $\beta \in \text{Br}\, K$ and a function assigning to each codimension one prime $\frakp \triangleleft R$ a positive integer $g_{\frakp}$ which equals one for all but finitely many $\frakp$. In particular, the localised Brauer class of the normal order $B$ above is $(\beta_A,g_{\frakp})$ in the notation of the previous paragraph.
\end{definition}

We now give an instance where the localised Brauer class and $R$-rank of a normal $R$-order determines the isomorphism class of the order.

\begin{assumption}  \label{ass:nicenormalorder}
Suppose now that $R$ is a Hensel local two-dimensional normal domain. Let $\Delta$ be a maximal $R$-order in a division ring and suppose that there exists a normal element $z \in \Delta$ such that 
\begin{enumerate}
    \item The quotient $\Delta/z\Delta$ is supported, as an $R$-module on a codimension one prime $\frakq$
    \item The element $z$ generates the radical of $\Delta_{\frakq}$
    \item $\Delta/z\Delta$ is hereditary.
\end{enumerate}
\end{assumption}
Under these assumptions, we construct the order
\begin{equation}  \label{eq:defDeltad}
\Delta_d = \Delta_d(z) := 
\begin{pmatrix}
\Delta & \Delta & \cdots & \Delta \\
 z\Delta & \Delta & & \vdots \\
\vdots & \ddots & \ddots & \vdots \\
z \Delta & \cdots & z \Delta & \Delta
\end{pmatrix}
\subseteq M_d(\Delta)    
\end{equation}
\begin{definition}  \label{def:triangularMatrix}
We will refer to the subalgebra $\Delta_d$ in (\ref{eq:defDeltad}) above as a {\em triangular modulo $z$ matrix algebra}.
\end{definition}
\begin{proposition}
\label{prop:Deltadisnormal}
Under Assumption~\ref{ass:nicenormalorder}, the order $\Delta_d$ is normal and has global dimension two. Furthermore, its localised Brauer class $(\beta,g_{\frakp})$ is given by
\begin{enumerate}
    \item $\beta$ is the Brauer class of $\Delta$.
    \item $g_{\frakq} = d$ and all other $g_{\frakp} = 1$.
\end{enumerate}
\end{proposition}
\begin{proof}
Note (1) follows from the fact that $\Delta_d$ is an order in $M_n(K\Delta)$. To check $\Delta_d$ is a reflexive $R$-module, note first that the $\Delta$ is reflexive being a maximal order. Also, $\Delta$ is a domain so $z$ must be a non-zero-divisor. Thus $z\Delta$ and hence also $\Delta_d$ are reflexive as well.

We consider now local structure at a codimension one prime $\frakp$. If $\frakp \neq \frakq$, then from Assumption~\ref{ass:nicenormalorder}(1), we know that $(\Delta_d)_{\frakp} \simeq M_d(\Delta_{\frakp})$ so is maximal. On the other hand, Assumption~\ref{ass:nicenormalorder} (2) ensures that $(\Delta_d)_{\frakq}$ is normal and $g_{\frakq} = d$. This completes the verificaton of (2). 

Finally, consider the normal element 
$$
t := 
\begin{pmatrix}
0 & 1 & 0 & \cdots & 0 \\
\vdots & 0 & \ddots & \ddots& \vdots \\
\vdots & \vdots & \ddots & \ddots &  0\\
0 & \vdots &  & \ddots & 1 \\
z  & 0 & \cdots & \cdots & 0
\end{pmatrix} \in \Delta_d.
$$
By Assumption~\ref{ass:nicenormalorder}(3), we see that $\Delta_d/t\Delta_d \simeq \prod_{i=1}^d \Delta/z\Delta$ is hereditary, so $\Delta_d$ itself must have global dimension two. 
\end{proof}
In the light of this proposition, we might consider $\Delta_d$ to be regular and almost maximal as in the title of this section. 

\begin{theorem}  \label{thm:Deltadzunique}
Suppose that Assumptions~\ref{ass:nicenormalorder} hold and let $\Delta_d$ be the normal $R$-order of (\ref{eq:defDeltad}). Let $\Lambda$ be any normal $R$-order with the same localised Brauer class as $\Delta_d$. Then $n = \frac{\deg \Lambda}{\deg \Delta_d}$ is an integer and $\Lambda \simeq M_n(\Delta_d)$ so has global dimension two.
\end{theorem}
\begin{proof}
Suppose $\deg \Delta_d | \deg \Lambda$ and let $\Lambda':= M_n(\Delta_d)$. Since the localised Brauer classes of $\Lambda, \Lambda'$ coincide, as do their degrees, we may embed them both in a common central simple $K$-algebra $Q$. By Corollary~\ref{cor:uniquenormal}, we know that $\Lambda_{\frakq} \simeq \Lambda'_{\frakq}$ so by altering one of the embeddings, we may suppose that we actually have $\Lambda_{\frakq} = \Lambda'_{\frakq}$.  Consider the $(\Lambda,\Lambda')$-bimodule  $B:=(\Lambda\Lambda')^{**} \subset Q$, the reflexive hull of the $R$-module $\Lambda\Lambda'$. Now $\Lambda'$ has global dimension two by Proposition~\ref{prop:Deltadisnormal} and $B$ is  Cohen-Macaulay as an $R$-module so is projective as a $\Lambda'$-module by \cite[Proposition~3.5]{Ramras} (the hypotheses are stated differently there but the proof applies in our case). 

We first show an isomorphism of right $\Lambda'$ modules, $B_{\Lambda'} \simeq \Lambda'_{\Lambda'}$. Now $R$ is Henselian so Krull-Schmidt holds for $\Lambda'$-modules. The indecomposable projective $\Lambda'$-modules are isomorphic to summands of $\Lambda'$ and there are exactly $d$ isomorphism classes of these, say $P_1, \ldots, P_d$, corresponding to the rows of $\Delta_d$. From Proposition~\ref{prop:cyclesimples} there exists an integer $r$ such that $(P_i)_{\frakq}/(P_i)_{\frakq} (\text{rad}\, \Lambda'_{\frakq}) \simeq S_i^{\oplus r}$ where $S_i$ is a simple $\Lambda'_{\frakq}$ module and that furthermore, the $S_i$ are all non-isomorphic. In particular, two finitely generated projective $\Lambda'$-modules are isomorphic, if and only if their localisations at $\frakq$ are isomorphic. Now by our choice of embeddings, $B_{\frakq} = \Lambda'_{\frakq}$ so $B \simeq \Lambda'$ as desired. 

To complete the proof of the theorem when $\deg \Delta_d | \deg \Lambda$, it suffices to show that the natural map $\Lambda \to \End_{\Lambda'} B = \Lambda'$ is an isomorphism. Since both sides are reflexive, this can be checked on codimension one primes $\frakp$. When $\frakp \neq \frakq$, it is an isomorphism since $\Lambda$ is maximal. When $\frakp = \frakq$, it is an isomorphism since we recalibrated so $\Lambda_{\frakq} = \Lambda'_{\frakq} = B_{\frakq}$. 

If $\deg \Delta_d$ does not divide $\deg \Lambda$, we apply the special case proved to $M_{\deg \Delta_d}(\Lambda)$, which shows that at least $\Lambda$ is Morita equivalent to $\Delta_d$. Hence $\Lambda \simeq \End_{\Delta_d} P$ for some projective $\Delta_d$-module. We can argue as before, looking locally at $\frakq$ to see that $d$ indecomposable projective modules occur equally in the decomposition of $P$, so $P \simeq \Delta_d^{\oplus n}$ for some $n$ as desired. 
\end{proof}

\section{toral terminal orders, regular centre case} 
\label{sec:terminalregular}
We have a series of concepts that are motivated by toric or toroidal geometry, but that does not satisfy either of these definitions.  We call these {\it toral}.
Let $(R,\frakm)$ be a two-dimensional noetherian regular Hensel local domain with field of fractions $K$. We introduce the notion of a {\em toral terminal} localised Brauer class on $R$.  These are all terminal. 
 Futhermore, in the special setting of \cite{CI24}, i.e.~when $R$ is an arithmetic surface with finite residue field and the ramification data are all $p$-torsion for some prime $p$, this  exactly agrees with terminal. Assuming that $R$ has enough roots of unity and a trivial Brauer group, we then classify the normal $R$-orders associated to toral terminal localised Brauer classes and show they all are regular in the sense that they have global dimension two. 

In the regular centre case, toral terminal localised Brauer classes fall into two types. The first, without secondary ramification is defined below.

\begin{definition}  \label{def:stableno2ndram}
A localised Brauer class $(\beta, g_{\frakp})$ on $R$ is {\em toral terminal without secondary ramification} if there exists a regular system of parameters $u,v \in R$ such that the following holds
\begin{enumerate}
    \item $\beta$ is unramified at every codimension one prime $\frakp$ except possibly $\frakp = (u)$ (that is, $a_{\frakp}(\beta) = 0$ for $\frakp \neq (u)$) and,
    \item All $g_{\frakp} = 1$ except possibly $\frakp = (v)$. 
\end{enumerate}
\end{definition}

We will construct orders via symbols and so need to assume the existence of enough roots of unity.
\begin{definition} \label{def:symbol}
Suppose $\zeta \in R$ is a primitive $n$-th root of unity. Given $a,b \in R-0$ we define the {\em $R$-symbol} $(a,b):=(a,b)^R_{\zeta}$ to be the $R$-algebra
$$
\Lambda = \frac{R\langle x,y \rangle}{(x^n-a, y^n-b, yw-\zeta xy)}.
$$
\end{definition}

\begin{proposition}  \label{prop:structuretoralterminalregno2nd}
Let $R$ be a two-dimensional regular noetherian Hensel local domain and $(\beta,g_{\frakp})$ be a toral terminal localised Brauer class with ramification as given in Definition~\ref{def:stableno2ndram}. Suppose that $\text{Br}\, R/\frakm = 0$ and $R$ possesses a primitive $n$-th root of unity $\zeta$ where $n$ is the order of $\beta$ in the $\text{Br}\, K$. 

Let $a\in R$ be any element chosen so the ramification of $\beta$ along $(u)$ is given by adjoining an $n$-th root of $a$. Then 

\begin{enumerate}
    \item $\Delta = (u,a)^R_{\zeta}$ is a maximal order in a division ring with the same ramification data as $\beta$. 
    \item Any normal order $\Lambda$ with localised Brauer class $(\beta,g_{\frakp})$ is isomorphic to $M_n(\Delta_d(v))$ (see notation in (\ref{eq:defDeltad})) where $d = g_{(v)}$.
\end{enumerate}
In particular, $\Lambda$ has global dimension two. 
\end{proposition}
\begin{proof}
Note that as secondary ramification must cancel, the ramification of $\beta$ along $(u)$ must be given by an \'etale cyclic extension of $R/(u)$, say of degree $n$. Now the cyclic \'etale extensions of $R/\frakm, R/(u)$ and $R$ all coincide, so we may find $a \in R$ defining this ramification (by adjoining $\sqrt[n]{a}$). Also, since $\text{Br}\, R = \text{Br}\, R/\frakm = 0$, we know $n$ is the order of $\beta$. 

Note that $\Delta = (u,a)^R_{\zeta}$ is Azumaya on the open set $u \neq 0$ and is maximal at the generic point $\mathfrak{p}$ of $u=0$ and has the same ramification as $\beta$ at $\mathfrak{p}$. Since $\Delta$ is also reflexive, it is a maximal order which has the same ramification as that of $\beta$. Furthermore, as already observed, $\text{Br}\, R = 0$ so both $\beta$ and $\Delta$ determine the same Brauer class in $\text{Br}\, K$. 

We seek now to apply Theorem~\ref{thm:Deltadzunique}. We begin by verifying Assumption~\ref{ass:nicenormalorder} for $z = v$. Note first that $\Delta$ is a domain since its degree coincides with the period. Clearly $\Delta/ v \Delta$ is supported along the prime $(v)$ only, and in fact $v$ generates the radical of the localisation $\Delta_{(v)}$. It remains to verify Assumption~\ref{ass:nicenormalorder}(3). Let $x \in \Delta$ be the $n$-th root of $u$ as in Definition~\ref{def:symbol}. Then $x$ gives a non-zero-divisor in $\Delta / v \Delta$.  Furthermore, $\Delta /(x,v)$ is the separable extension $(R/\frakm)(\sqrt[n]{a})$ so $\Delta/ v \Delta$ is indeed hereditary. 
\end{proof}

\begin{definition} \label{def:stable2ndram}
A localised Brauer class $(\beta,g_{\frakp})$ on $R$ is {\em toral terminal with secondary ramification} if there exists a regular system of parameters such that the following holds:
\begin{enumerate}
    \item $\beta$ is unramified away from $(uv)$.
    \item $\beta$ is ramified along both $(u)$ and $(v)$ and the ramification at these prime ideals are given by totally ramified field extensions of the residue fields.
    \item all $g_{\frakp} = 1$ except possibly $g_{(v)}$.
\end{enumerate}
More generally, (but still assuming $R$ regular), we say a localised Brauer class $(\beta,g_{\frakp})$ is {\em toral terminal} if it is either toral terminal with secondary ramification as above, or toral terminal without secondary ramification as in Definition~\ref{def:stableno2ndram}.
\end{definition}

\begin{proposition}  \label{prop:structurestable2ndram}
Let $R$ be a two-dimensional regular noetherian Hensel local domain and $(\beta,g_{\frakp})$ be a toral terminal localised Brauer class with secondary ramification as given in Definition~\ref{def:stable2ndram}. Suppose that $\text{Br}\, R/\frakm = 0$ and $R$ possesses a primitive $n$-th root of unity where $n$ is the order of $\beta$ in $\text{Br}\, K$. Any normal order $\Lambda$ with localised Brauer class $(\beta,g_{\frakp})$ is isomorphic to $M_n(\Delta_d(y))$ (see notation in (\ref{eq:defDeltad})) where
\begin{enumerate}
    \item $d = g_{(v)}$.
    \item $\Delta_d(y)$ is built from the maximal order $$\Delta = (au,bv)^R_{\zeta}, $$
    where $a,b \in R^{\times}$ are units, $\zeta$ is an appropriate $n$-th root of unity and $y$ is the $n$-th root of $v$ used in Definition~\ref{def:symbol}.
\end{enumerate}
In particular, $\Lambda$ has global dimension two. 
\end{proposition}
\begin{proof}
We use Theorem~\ref{thm:Deltadzunique} along the same lines as the proof of Proposition~\ref{prop:structuretoralterminalregno2nd}. 
It suffices to find $u,v, \zeta$ such that $(u,v)^R_{\zeta}$ has the same ramification as $\beta$. 

Let $\kappa_u$ denote the residue field at the point $(u)$. The ramification $a_{(u)}(\beta) \in H^1(\kappa_u, \bQ/\bZ)$ of $\beta$ along $(u)$ corresponds to a cyclic field extension $\tilde{\kappa}/\kappa_u$ and a generator of the Galois group. Since we assumed existence of primitive $n$-th roots of unity, we may use Kummer theory to see that $\tilde{\kappa} = \kappa_u(\sqrt[n]{\bar{v}})$ for some $\bar{v} \in \kappa_u$. Since this is a totally ramified extension, we may change generators and assume that $\bar{v}$ is the restriction of $bv$ for some $b \in R^{\times}$. If $\sigma$ is the chosen generator of the Galois group, then $\sigma(\sqrt[n]{bv}) = \zeta \sqrt[n]{bv}$ where $\zeta$ is the $n$-th root of unity required in (2) above. Arguing the same way for ramification of $\beta$ along $(v)$ and using the fact that secondary ramification cancels, we see that $a_{(v)}(\beta)$ is given by adjoining an $n$-th root of $au$ for some $a \in R^{\times}$ and we are done. 
\end{proof}

\section{Hirzebruch-Jung singularities as cyclic quotient singularities} \label{sec:HJcyclic}

Over the complex numbers the Hirzebruch-Jung singularities are well understood and all arise as cyclic quotient singularities. In this section, we show a similar result for two-dimensional normal singularities which are Hirzebruch-Jung in the sense that their minimal resolution is a string of projective lines defined over the residue field. 

More precisely, suppose that $R$ is a two-dimensional normal noetherian excellent commutative Hensel local domain with residue field $\kappa$. Suppose it is a {\em $\kappa$-rational Hirzebruch-Jung} singularity in the sense that it has a rational minimal resolution $f \colon X \to \Spec R$ such that the exceptional locus is a string $E_1, \ldots E_r$ of exceptional curves isomorphic to the projective line over $\kappa$, that is, all  $E_i \simeq \bP^1_{\kappa}$, $E_i$ intersects $E_{i+1}$ in a single point which is $\kappa$-rational and there are no other intersections. There are more complicated analogues of Hirzebruch-Jung singularities studied in the literature which we have not studied as they do not arise in the study of the terminal orders considered in this paper. 

Since $R$ is Hensel local, we may choose irreducible curves $E_0, E_{r+1} \subset Y$ such that $E_0$ (respectively $E_{r+1}$) intersects $E_1$ (respectively $E_r$) in a single $\kappa$-rational point and furthermore, $E_0$ and $E_2$ (respectively $E_{r+1}$ and $E_{r-1}$) are disjoint. Let $m_i = - E_i^2$. We define pairs $\nu_0 = (0,1), \nu_1 = (1,0)$ and then recursively 
\begin{equation}  \label{eq:definenu}
\nu_{i+1} = m_i \nu_i - \nu_{i-1}, \quad \text{for} \ i = 1, \ldots r. 
\end{equation}
When $R$ is defined over the complex numbers, these give the exceptional curves in the toric description of the Hirzebruch-Jung singularity. An easy induction shows that the dot product $(1,1).\nu_i$ weakly increases with $i$, so $(m,-k):= \nu_{r+1}$ satisfies $0<k<m$. Similarly, one can show that $k,m$ are relatively prime. The integer $m$ appears in our key theorem below.

\begin{theorem}  \label{thm:HJiscyclic}
Let $R$ be a $\kappa$-rational Hirzebruch-Jung singularity as defined above and $X = \Spec R$. The Weil divisor $f_*E_0$ is $m$-torsion in the sense that $mf_*E_0$ is Cartier. Furthermore, there is a natural ring structure on 
$$ S := \bigoplus_{l=0}^{m-1} \cO_X(-lf_*E_0) t^l$$
making it a regular local ring with the same residue field $\kappa$ as $R$. In particular, if $R$ contains $m$-th roots of unity, then $R$ is a cyclic quotient singularity. Furthermore, $S/R$ is \'etale away from the singular point. 
\end{theorem}
Before launching in to the proof, we set up the appropriate theory first. The idea is to recover as much of the toric theory of Hirzebruch-Jung singularities over the complex numbers as possible. To this end we consider
\begin{definition}  \label{def:toral}
A divisor $D$ on $Y$ is {\em toral} if it belongs to $\bigoplus_{i=0}^{r+1} \bZ E_i$. We say $w \in R$ is {\em toral} or is a {\em toral function} if the associated divisor of $f^*w$ is toral.
\end{definition}
\begin{remark}
Note that the definition depends on the choice of $E_0$ and $E_{r+1}$. 
\end{remark}
Our first order of business is to classify all toral functions and show their divisors on $Y$ are given by the lattice points in the cone $\bR_{\geq 0} (0,1) + \bR_{\geq 0} (m,-k)$ where we recall $(m,-k) = \nu_{r+1}$. Note that any function $w \in R$ is determined, up to $H^0(\cO_Y^{\times}) = R^{\times}$, by their divisor $(f^*w)$ on $Y$. It thus suffices to classify effective toral divisors $D$ on $Y$ such that $D \sim 0$. Now $f$ is a rational resolution, so by \cite[Proposition~11.1 i)]{Lip}, we know $D \sim 0$ if and only if $D.E_i = 0$ for $i = 1,\ldots, r$. 

To enumerate all such toral divisors, we work as follows. Let $L = \bigoplus_{i=0}^{r+1} \bZ E_i$ and consider $L^* := \Hom_{\bZ}(L,\bZ) = \bigoplus \bZ E_i^*$ where $\{E_i^*\}$ is a dual basis to the $E_i$. For $i=1,\ldots, r$, let 
$E_i^{\vee}:=(C \mapsto E_i.C) \in L^*$ and $\bE < L^*$ be the subgroup generated by the $E_i^{\vee}$. The next result follows from equation (\ref{eq:definenu}).
\begin{proposition}  \label{prop:nufunction}
The homomorphism $\nu \colon L^* \to \bZ^2$ defined by  $E_i^* \mapsto \nu_i$ is surjective with kernel $\bE$. 
\end{proposition}
The $D\sim 0$ condition thus corresponds to the fact that the naturally induced map $ L^* \to \bZ$ given by  $\chi \mapsto \chi(D)$ actually factors through $\nu$ to give a homomorphism $\lambda_D \colon \bZ^2 \to \bZ$. Note that 
\begin{equation}  \label{eq:coeffD}
 D = \sum_{i=0}^{r+1} \lambda_D (\nu_i) E_i.
\end{equation}
Suppose now that $\lambda_D$ is given by dot product with $(i,j) \in \bQ^2$. Now $\nu_0 = (0,1), \nu_1 = (1,0)$ so $i,j \in \bZ$. The divisor $D$ corresponding to $(i,j)$ is effective when $(i,j).\nu_0\geq 0, (i,j).\nu_{r+1} \geq 0$, that is $(i,j)$ lies in the cone $\bR_{\geq 0} (1,0) + \bR_{\geq 0} (k,m)$. We now abuse notation and write $x^iy^j$ for any toral function with this $D$ as its divisor in $Y$. The notation allows us to write $x^iy^jx^{i'}y^{j'} = x^{i+i'}y^{j+j'}$ with the caveat that it holds only modulo $R^\times$. We sumarize the results up to this point.

\begin{proposition}  \label{prop:toralfunctions}
The toral functions are $x^iy^j$ where $(i,j) \in \bR_{\geq 0} (1,0) + \bR_{\geq 0} (k,m)$. Its divisor in $Y$ is $\sum \lambda_l  E_l$ where $\lambda_l = (i,j). \nu_l$. 
\end{proposition}

\begin{corollary}  \label{cor:mtorsionWeildiv}
The Weil divisor $f_*E_0$ is $m$-torsion.
\end{corollary} 
\begin{proof}
The divisor of zeros $\sum \lambda_l  E_l$ of the toral function $x^ky^m$  has $\lambda_0 = m, \lambda_{r+1} = 0$.
\end{proof}

Let $\frakm$ be the maximal ideal of $R$ and $Z$ be the fundamental cycle so $\frakm^n = f_* \cO(-nZ)$ \cite[Lemma~9.4.14]{Liu02}. Note that $Z = E_1 + \ldots + E_r$. Consider a Weil divisor $C \subset \Spec R$ so $\cO(-C)$ is a reflexive ideal. Note that $f^*\cO(-C)/\cT \simeq \cO_Y(-\tilde{C})$ for some $\frakm$-torsion sheaf $\cT$ and Cartier divisor $\tilde{C}$ which is the strict transform of $C$ away from the exceptional locus. 
\begin{definition} \label{def:induceddiv}
We call $\tilde{C}$ the {\em pullback} of $C$. This agrees with the usual pullback in the case that $C$ is Cartier.
\end{definition}

Note that reflexivity ensures that $f_*\cO_Y(-\Ctilde) = \cO(-C)$ so $\cO(-C)$ is {\em contracted} in the language of \cite[Definition~6.1]{Lip}. Suppose more generally that $D \in \text{Div} Y$ is such that $\cO_Y(-D)$ is generated by global sections so \cite[Corollary to 7.3]{Lip} ensures that $\frakm^n f_*\cO_Y(-D) = f_*\cO_Y(-D - nZ)$. 
Applying $f_*$ to the exact sequence
$$
0 \to \cO_Y(-D - Z) \to \cO_Y(-D) \to \cO_Z(-D) \to 0 
$$
gives the exact sequence
$$
0 \to f_*\cO(-D) \otimes_R R/\frakm \to H^0(\cO_Z(-D)) \to R^1f_* \cO_Y(-D - Z).
$$
Now $\cO_Y(- D - Z)$ is generated by global sections since the same is true of $\cO_Y(-D)$ and $\cO_Y(-Z)$, so $R^1f_* \cO_Y(-D - Z)) = 0$ from which follows the next result.
\begin{lemma}  \label{lem:reflexivegen}
Let $D$ be a divisor on $Y$ such that $\cO_Y(-D)$ is generated by global sections. Write $I = f_*\cO_Y(-D)$ which will be an ideal in $R$ if $D$ is effective. Then we have $I \otimes_R R/\frakm \simeq H^0(\cO_Z(-D))$. In particular, a set of generators for $I$ can be found by giving a set of global sections of $\cO_Y(-D)$ whose restriction to $Z$ gives a spanning set for $H^0(\cO_Z(-D))$. This applies in particular to $I = \cO(-C)$ where $C$ is an effective Weil divisor on $X$ and $D = \Ctilde$ is the pullback. 
\end{lemma}

We will use this lemma to find toral generators for $\cO(-if_*E_0)$. First, we need a ``toral' basis for $H^0(Z, \cL)$ where $\cL$ is a line bundle on $Z$ with all $d_i := \deg_{E_i} \cL \geq 0$. 
\begin{definition}  \label{def:basictoral}
The intersections $E_i \cap E_{i+1}, i=0, \ldots, r$ are said to be the {\em toral points} of $Z$. A non-zero section $s \in H^0(Z, \cL)$ is {\em toral} if its zero set is a union of exceptional curves and toral points. Furthermore, we say $s$ is {\em basic toral} if it also satisfies the following condition: whenever $s|_{E_i} \neq 0$ but has a zero at $E_i \cap E_{i+1}$ (respectively $E_i \cap E_{i-1}$), then $s|_{E_j} = 0$ for $j>i$ (respectively $j<i$).  
\end{definition}
We can construct basic toral sections as follows. Start with some non-zero section $s_i \in H^0(E_i, \cL|_{E_i}) \simeq H^0(\bP^1, \cO(d_i))$ which is ``toral'' in the sense that its zeros are confined to $E_{i-1} \cup E_{i+1}$. Note that up to a scalar in $\kappa$, there are $d_i + 1$ of these. We show it can be extended uniquely to a basic toral section $s$ of $\cL$. Now if $s_i$ has a zero at $E_{i-1} \cap E_i$, then we simply extend by setting $s|_{E_j} = 0$ for $j < i$. If on the other hand $s_i$ is non-zero at $E_{i-1} \cap E_i$, then there is a unique way to extend it to a toral section on $E_{i-1}$ and we can continue by induction. A similar argument determines $s$ on $E_j$ for  $j>i$.
This gives the following

\begin{lemma}  \label{lem:classifybasictoral}
Any basic toral section $s$ is uniquely determined by any non-zero restriction  $s|_{E_l}$ and has the form constructed in the preceding paragraph. 
\end{lemma}

\begin{proposition}  \label{prop:toralbasisexists}
Given a line bundle $\cL$ on $Z$ with non-negative degrees $d_i:= \deg_{E_i} \cL \geq 0$, there exists a basis for $H^0(Z, \cL)$ consisting of basic toral sections. Furthermore, this basis is unique up to scaling the basis elements. 
\end{proposition}
\begin{proof}
Note firstly that $H^0(Z,\cL)$ is naturally isomorphic to the kernel of the natural map $\oplus_{l=1}^r H^0(E_l, \cO_{E_i}(d_i)) \to H^0(T,\cO_T)$ where $T$ is the set of nodes in $Z$. This has dimension $d = \sum (d_i + 1) - (r-1)$. Note that the only linear relations between basic toral sections are those which are scalar multiples of each other so it suffices to find $d$ basic toral sections, no two of which are multiples of each other. The above construction provides these once we note that the basic toral section constructed from some toral section $s_i \in H^0(E_i,\cL|_{E_i})$ coincides with one constructed from $s_{i-1} \in H^0(E_{i-1},\cL_{E_{i-1}})$ if and only if $s_i, s_{i-1}$ take on the same non-zero value at $E_{i-1} \cap E_i$. 
\end{proof}

We can now construct toral generators for reflexive ideals in $R$. 

\begin{proposition}  \label{prop:toralgenforreflexive}
Let $D$ be an effective toral divisor on $Y$ such that $\cO_Y(-D)$ is generated by global sections and $I = f_*\cO_Y(-D)$ be the associated ideal of $R$. Then $I$ is generated by toral functions.  
\end{proposition}
\begin{proof}
Combining Lemma~\ref{lem:reflexivegen} with Proposition~\ref{prop:toralbasisexists}, it suffices to lift every basic toral section of $\cO_Z(-D)$ to a toral section of $\cO_Y(-D)$. Let $d_i := -D.E_i$ and consider a basic toral section whose restriction to $E_i$ has a zero of order $e$ at $E_{i-1}$ and order $d_i - e$ at $E_{i+1}$. Lifting this to a toral function amounts to finding an effective toral divisor $\Delta = \sum_{j = 0}^{r+1} \delta_j E_j$ such that a) $-D \sim \Delta $ and b) $
\delta_{i-1} = e, \ \delta_i =0, \ \delta_{i+1} = d_i -e$. Condition a) amounts to checking that all the intersection numbers $(D + \Delta).E_j = 0$. We solve these equations for $\delta_j$ by induction on $|j-i|$ and simultaneously prove effectivity of $\Delta$ by proving $\delta_j$ is non-decreasing for $j\geq i$ and non-increasing for $j \leq i$. The base case is satisfied since
$$ (D+\Delta).E_i = \delta_{i-1} + D.E_i + \delta_{i+1} = e - d_i + (d_i -e) = 0.$$
Suppose now that non-decreasing integers  $\delta_i,\ldots,\delta_j$ have now been defined satisfying $(D+\Delta).E_l = 0$ for $l=i,\ldots, j$. We examine the equation
$$
0 = (D+\Delta).E_j = -d_j + \delta_{j-1} - m_j \delta_j + \delta_{j+1}.
$$
We may thus solve for the integer $\delta_{j+1}$ which further satisfies
$$
\delta_{j+1} - \delta_j = d_j + [(m_j-1)\delta_j - \delta_{j-1}].
\geq 0$$
by the inductive hypothesis. A similar argument solves for non-increasing $\delta_j$ when $j \leq i$. 
\end{proof}

\begin{proof} of Theorem~\ref{thm:HJiscyclic} Inspired by Corollary~\ref{cor:mtorsionWeildiv} or rather, its proof, we define a ring structure on 
$$
S = \bigoplus_{l=0}^{m-1} \cO(-lf_*E_0)t^l
$$
by defining $t^{-m} = x^ky^m$. Given a toral function $x^iy^j \in \cO(-lf_*E_0)$, we say 
$$
x^iy^jt^l = x^{i-kl/m}y^{j-l}
$$
is a toral function in $S$. It suffices to find two toral functions in $f_1,f_2 \in S$ which generate the maximal ideal 
$$ \frakn := \frakm \oplus \cO(-f_*E_0)t \oplus \ldots \oplus \cO(-(m-1)f_*E_0)t^{m-1}.$$
By Proposition~\ref{prop:toralgenforreflexive}, it suffices to show $f_1,f_2$ will generate all the toral elements in the summands $\frakm, \ldots ,\cO(-(m-1)f_*E_0)t^{m-1}$. Note also that our sloppiness in notation for $x^iy^j$ is warranted since we only care about the ideal generated by $f_1,f_2$. 

From Proposition~\ref{prop:toralfunctions}, we know that $x^iy^j$ is a toral function in $\cO(-lf_*E_0)$ if and only if $(i,j) \in \bR_{\geq 0} (1,0) + \bR_{\geq 0} (k,m)$ and furthermore $l\leq (i,j).\nu_0 = j$. We may thus let
\begin{equation} \label{eq:f1generator}
 f_1 = x^ky^mt^{m-1} = x^{k/m}y.    
\end{equation}
To find $f_2$ we first find $l \in \{1,\ldots,m-1\}$ which solves $kl \equiv -1 \mod m$, which is possible since $k$ and $m$ are relatively prime. Let $i = \frac{kl+1}{m}$ and 
\begin{equation}  \label{eq:f2generator}
f_2 = x^iy^lt^l = x^{1/m}.
\end{equation}
An elementary calculation shows that the toral elements of $S$ all have the form $x^iy^j$ where $i \in \frac{1}{m}\bZ, j \in \bZ$ and furthermore $0 \leq j \leq \frac{m}{k}i$. It follows that all the toral elements in $\frakn$ are generated by $f_1$ and $f_2$. 

Finally, the construction of $S$ here is the cyclic covering trick, see for example~\cite[4.1.B]{L04}, which away from the singularity uses an $m$-torsion line bundle so $S/R$ is \'etale. 
\end{proof}

For use in the next section, we record the following fact which follows from Proposition~\ref{prop:toralfunctions}.
\begin{lemma}  \label{lem:mthPowerf1}
The toral function $f_1 \in S$ defined in Equation~\ref{eq:f1generator} is such that $f_1^m\in R$ and its divisor is $mf_*E_0$. 
\end{lemma}

\section{toral terminal orders, singular centre case}

Unlike in the geometric case where the residue fields are algebraically closed, there are now terminal orders with singular centre \cite{CI24}. Their ramification data were classified in the case where the ``index'' \cite[\S 3,p.~6]{CI24} was a prime $m>5$. Here we classify the corresponding orders, giving explicit constructions of them. 

Throughout this section, we let $(R,\frakm)$ be an excellent two-dimensional noetherian Hensel local domain with residue field $\kappa$. The classification of terminal ramification data on $R$ is best encapsulated via the following definition.

\begin{definition}  \label{def:terminalHJ}
A localised Brauer class $(\beta,g_{\frakp})$ on $R$ is {\em toral terminal} if either a) $R$ is regular and we are in the case of Definition~\ref{def:stableno2ndram} or \ref{def:stable2ndram}, or if b) the following hold:
\begin{enumerate}
    \item $R$ is a $\kappa$-rational Hirzebruch-Jung singularity whose residue field has trivial Brauer group. Let $E_1, \ldots, E_r$ be the string of exceptional curves in the minimal resolution (indexed naturally so $E_i$ intersects $E_{i+1}$). 
    \item $\beta$ is unramified along codimension one primes in $R$.
    \item The order $m$ of $\beta$ equals the {\em determinant} of $R$ which is defined to be $\det R := \det(M_R)$ where
    $$ M_R := -
    \begin{pmatrix}
    E_1^2 & 1 & 0 & \ldots & 0 \\
    1 & E_2^2 & \ddots & \ddots & \vdots \\
    0 & \ddots & \ddots & \ddots & 0 \\
    \vdots & & & & 1 \\
    0 & \ldots & 0 & 1 & E_r^2 
    \end{pmatrix}.
    $$
    \item At most one $g_{\frakp}\neq 1$ in which case, $\frakp$ corresponds to an irreducible curve $C$ on the minimal resolution which (scheme-theoretically) intersects the exceptional curve in a single $\kappa$-rational point $y \notin E_2 \cup \ldots \cup E_{r-1}$.
\end{enumerate}

\begin{remark}  \label{rem:relationToTerminal}
\begin{enumerate}
    \item Given a toral terminal localised Brauer class $(\beta,g_{\frakp})$ on $R$ as above, then \cite[Theorem~7.1]{CI24} shows that it is terminal whenever $m$ is a prime $>2$. Furthermore, if $m$ is a prime $>5$ and $\kappa$ is finite, then these are the only terminal localised Brauer classes on singular $R$.   
    \item In \cite{CI24}, the residue fields are all assumed to be finite and so have trivial Brauer group.
\end{enumerate}
\end{remark}
\end{definition}

Let $(R,\frakm)$ be a $\kappa$-rational Hirzebruch-Jung singularity of determinant $m$ (see Definition~\ref{def:terminalHJ}(3)). Suppose that $\kappa$ contains a primitive $m$-th root of unity. From Theorem~\ref{thm:HJiscyclic}, there exists a regular local ring $(S,\frakn)$ such that $S/\frakn = \kappa$ and $S/R$ is a cyclic extension which is \'etale away from the singular point. In Theorem~\ref{thm:HJiscyclic}, it is presented as a $\bZ/m$-graded algebra $S = \oplus_{i \in \bZ/m} S_i$ so $(\bZ/m)^{\vee} = \mu_m$ acts on it naturally. Let $\alpha \in H^1(\kappa,\bZ/p)$ correspond to a cyclic degree $m$ field extension $\tilde{\kappa}/\kappa$ with some chosen action of $\bZ/m$. Since $R$ is Hensel local, there is a corresponding cyclic \'etale extension $\tilde{R}/R$ and a corresponding $\mu_m$-graded decomposition $\tilde{R} = \oplus_{\omega \in \mu_m} \tilde{R}_{\omega}$. 

\begin{definition}  \label{def:symbolHJ}
We define the {\em symbol} $(S,\alpha)$ to be the $R$-algebra whose underlying $R$-module structure is given by 
$$ \Delta := S \otimes_R \tilde{R}$$
and multiplication given by the skew-commutation relations
$$ rs = \omega^i sr\quad \textup{for all} \ \ s \in S_i, \ r \in \tilde{R}_{\omega}.$$
\end{definition}

\begin{proposition}  \label{prop:isOrderInDivision}
The symbol $\Delta = (S,\alpha)$ defined above is a maximal order in a division ring. Furthermore, $\Delta$ is Azumaya in codimension one. 
\end{proposition}
\begin{proof}
Note that the commutative algebra $S \otimes_R \tilde{R}$ is an \'etale extension of $S$ and hence regular and thus Cohen-Macaulay. It follows that $\Delta$ is a reflexive $R$-module. In codimension one, both $S/R$ and $\tilde{R}/R$ are \'etale so $\Delta$ is defined using the usual symbol construction of Azumaya algebras. We see thus that $\Delta$ is a maximal order, Azumaya in codimension one. 

It only remains to show that $\Delta_K := \Delta\otimes_R K(R)$ is a division ring, which we do by showing that it cannot have  period $<m$. Suppose that $\tilde{R}$ is obtained by adjoining an $m$-th root of $\alpha \in R^{\times}$ to $R$. Suppose that the order of $\Delta_K$ is $n|m$. If $\sigma$ denotes the action of some fixed primitive $p$-th root of unity on $S$, then the cyclic algebra $A:= K(S)[z;\sigma]/(z^m - \alpha^n)$ is a full matrix algebra over $K(R)$. We see thus from \cite[Corollary~4.7.5]{GilSz} that $\alpha^n \in K(R)$ is a norm from $K(S)$, say $\alpha^n = N(\beta)$ where $\beta = \beta_1\beta_2^{-1}$ for $\beta_1,\beta_2 \in S$. Now $S$ is a UFD so we may prime factorise both $\beta_1$ and $\beta_2$. 

We first show that by modifying $\beta$ by a 1-coboundary $\gamma^{-1}\sigma(\gamma), \gamma \in K(S)$, we may assume that $\beta \in S^{\times}$. Indeed, note that $N(\beta_1),N(\beta_2)$ differ by the unit $\alpha^n\in R^{\times}$. Thus if there is any prime factor $p_2|\beta_2$, there is some prime factor $p_1 | \beta_1$ and $i$ such that $p_2 | \sigma^i(p_1)$. We may thus multiply by some 1-coboundary so that these factors now cancel. Having reduced the number of prime factors of $\beta_2$, we are done by induction.

We now use the fact that $S/\frakn = R/\frakm$ to see that modulo $\frakm$, $\alpha^n$ is a $m$-th power. Since $\tilde{\kappa}$ is a degree $m$ extension of $\kappa$ obtained by adjoining a $m$-th root of $\alpha$, we must have $n=m$. 
\end{proof}

\begin{proposition} \label{prop:DeltaSameRam}
Let $(\beta,g_{\frakp})$ be a toral terminal localised Brauer class on a $\kappa$-rational Hirzebruch-Jung singularity $R$. Suppose that $R$ has a primitive $m$-th root of unity where $m$ is the order of $\beta$. Then there is an $\alpha \in H^1(\kappa, \bZ/m)$ such that class of the symbol $(S,\alpha)$ in $\textup{Br}\,K(R)$ is $\beta$. 
\end{proposition}
\begin{proof}
We use the notation in Definition~\ref{def:terminalHJ}. Let $X \to \Spec R$ be the minimal resolution of the $\kappa$-rational Hirzebruch-Jung singularity $R$. Note that $\beta$ is ramified only along the exceptional curve, so by the Artin-Mumford-Saltman sequence \cite{AM}, \cite[Theorem~6.12]{Salt08}, the ramification covers of the $E_i$ are all \'etale. Now $\textup{Br}\, \kappa = 0$ so the ramification along $E_i$ is given by cyclic cover of the form $\bP^1_{\kappa_i} \to \bP^1_{\kappa} \simeq E_i$ where $\kappa_i$ is a cyclic extension of $\kappa$ of degree $n|m$. The ramification is thus given by an element of $H^1(\kappa, \bZ/p)$.

We now use the theory developed in \cite[Setion~4]{CI24}. There was an assumption there that the residue field was finite, but the theory goes through in this case, as long as one realises that the absolute Galois group $G$ of $\kappa$ is now not necessarily $\hat{\bZ}$. In particular, we have the following version of \cite[Proposition~9.8]{CI24}.
\begin{lemma}
There exists an homomorphism $z\colon \bZ^2 \to H^1(\kappa,\bZ/m)$ such that the ramification of $\beta$ along $E_i$ is given by $z(\nu_i)$ where $\nu_i$ is as defined in Equation~(\ref{eq:definenu}). 
\end{lemma}
In particular, we see that the ramification of $\beta$, and hence $\beta$ itself is completely determined by $z(0,1) = z(\nu_0) = 0$ and $z(1,0) = z(\nu_1)$, the ramification along $E_1$. It is now clear that we can pick $\alpha \in H^1({\kappa},\bZ/m)$ so that the symbol $(S,\alpha)$ has the same ramification as $\beta$ along $E_1$, and hence belongs to the same Brauer class over $K(R)$. 
\end{proof}

\begin{theorem}  \label{thm:mainHJ}
Let $\Lambda$ be a normal order over an excellent two-dimensional Hensel local noetherian domain $(R,\frakm)$ which is not regular. Suppose its localised Brauer class  $(\beta,g_{\frakp})$ is toral terminal. If $R$ has a primitive $m$-th root of unity where $m$ is the order of $\beta$, then $\Lambda \simeq M_n(\Delta_d(z))$ where 
\begin{enumerate}
    \item $\Delta$ is the symbol $(S,\alpha)$ where $S$ is the regular cyclic cover of $R$ constructed in Theorem~\ref{thm:HJiscyclic} and $\alpha \in H^1(\kappa, \bZ/m)$. 
    \item $n \in \bN$, $d$ is either 1 or the unique $g_{\frakp}$ not equal to 1, and $z \in S \subset \Delta$ is the normal element denoted $f_1$ in Equation~(\ref{eq:f1generator}). 
\end{enumerate}
In particular, $\Lambda$ has global dimension two. 
\end{theorem}
\begin{proof}
From Theorem~\ref{prop:DeltaSameRam}, we may choose $\alpha$ so that $\Delta = (S,\alpha)$ represents the Brauer class $\beta$. We also know from Proposition~\ref{prop:Deltadisnormal} that $\Delta$ is a maximal order in a division ring. The result will thus follow from Theorem~\ref{thm:Deltadzunique} once we verify Assumption~\ref{ass:nicenormalorder}. 
Let $f \colon X \to \Spec R$ be the minimal resolution. Using the notation in Section~\ref{sec:HJcyclic}, toral terminal implies that we may pick $E_0\subset X$ to be such that $C:=f_*E_0$ corresponds to the codimension one prime $\frakq$ with $g_{\frakq} \neq 1$ if such a prime exists (and is otherwise an arbitrary prime divisor intersecting $E_1\setminus E_2$ in a $\kappa$-rational point). 

Note that $z$ is a toral function and hence gives a normal element of $\Delta$. We also know from Lemma~\ref{lem:mthPowerf1} that $z^m \in R$ and that its associated divisor is $mf_*E_0$. It follows that $(z) \triangleleft S$ is the unique prime lying over $\frakq$. Thus $\Delta/z \Delta$ is supported on $C$ as an $R$-module and Assumption~\ref{ass:nicenormalorder}(1) is verified. It also follows that $z$ lies in the radical of $\Delta_{\frakq}$.
Consider now
$$
\bar{\Delta} := \Delta/ z \Delta \simeq 
S/(z) \otimes_{R/\frakq} \tilde{R}
$$
where $\tilde{R}$ is the cyclic \'etale extension of $R$ determined by $\alpha$. 
 To see that $z$ generates the radical of $\Delta_{\frakq}$ it suffices to observe that $\bar{\Delta}_{\frakq}$ is a central simple $K(R/\frakq)$-algebra since it is readily identified with a symbol. This completes the verification of Assumption~\ref{ass:nicenormalorder}(2) so it remains only to show that $\bar{\Delta}$ is hereditary. To this end, let $y$ be the other generator of $\textup{rad}\, S$ denoted $f_2$ in Equation~\ref{eq:f2generator}. It is normal in $\Delta$ and thus $\bar{\Delta}$. Then $\bar{\Delta}/y \bar{\Delta} \simeq \kappa \otimes_R \tilde{R}$ which is a field and hence has global dimension zero. This completes the proof of the theorem. 

\end{proof}

\begin{corollary}  \label{cor:terminalIsRegular}
Let $\Lambda$ be a normal order over an excellent two-dimensional normal noetherian Hensel local domain $R$ with finite residue field. Suppose that its localised Brauer class $(\beta,g_{\frakp})$ is terminal and the following hold 
\begin{enumerate}
    \item the order $m$ of $\beta$ is prime $>5$ and
    \item $R$ has primitive $m$-th roots of unity.
\end{enumerate}
Then $\Lambda$ has global dimension two. 
\end{corollary}

\section{Appendix}  \label{sec:appendix}

The theory of normal and more generally hereditary orders over a complete discrete valuation ring is well-known and can be found in standard texts such as Reiner's classic text \cite{Reiner}. In this appendix, we extend some results to arbitrary discrete valuation rings $R$ which are not necessarily complete. 

Let $\frakm$ be the maximal ideal of $R$ and $K$ be its field of fractions. Let $\Delta$ be a maximal order in some $K$-central division ring $K\Delta$. Let $\Lambda\subseteq M_n(\Delta)$ be an hereditary order in $M_n(K\Delta)$ with say Jacobson radical $J$. Note that $\Delta/\textup{rad}\, \Delta$ is central simple, say isomorphic to $M_r(D)$ where $D$ is a division ring. 

\begin{proposition}  \label{prop:uniformprojective}
Consider a right projective $\Lambda$-module $P$ such that $\textup{End}\, P \simeq \Delta$. Then $P/PJ \simeq S^{\oplus r}$ where $S$ is a simple $\Lambda$-module and $r$ is the integer such that $\Delta/\textup{rad}\, \Delta \simeq M_r(D)$ for some division ring $D$. This result holds in particular for $P = \Delta^n$.
\end{proposition}
\begin{proof}
Consider the natural ring homomorphism 
$$\Delta = \textup{End}_{\Lambda}  P \to \textup{End}_{\Lambda} P/PJ.$$
This map is surjective since $P$ is projective. From the ideal theory of maximal orders \cite[Theorem~2.3]{AG}, the only semisimple quotient of $\Delta$ is $\Delta/\textup{rad}\, \Delta \simeq M_r(D)$ so $P/PJ$ must be the direct sum of $r$ copies of a single simple.

Note finally that $\textup{End}_{\Lambda} \Delta^n = \textup{End}_{M_n(\Delta)} \Delta^n$ since $\Lambda$ is also an order in $M_n(K\Delta)$. The final statement follows thus from $\Delta = \textup{End}_{M_n(\Delta)} \Delta^n$ which is a consequence of Morita theory. 
\end{proof}

\begin{definition}  \label{def:normal}
We say that $\Lambda$ is {\em normal} if its Jacobson radical $J$ is free of rank 1 as a left and right module. 
\end{definition}
If $\Lambda$ is normal, then it is easy to see that one can choose a {\em uniformiser} $t \in J$ such that $J = \Lambda t = t \Lambda$ so the inner automorphism $r \mapsto trt^{-1}$ induces an automorphism of the semisimple ring $\Lambda/J$. We refer to this as the {\em $t$-action} on the Wedderburn components, which induces an analogous $t$-action on the simples (it maps a simple $S \mapsto S \otimes_{\Lambda} t\Lambda$). Suppose there are $d$ simples $S_1, \ldots, S_d$. The following show that the $t$-action permutes the Wedderburn components cyclically. 
\begin{proposition}  \label{prop:cyclesimples}
The action of $t$ permutes all the simples cyclically. In other words, by re-indexing if necessary, we may assume that $S_i^{\oplus r} \simeq Pt^{i-1}/Pt^{i}$ where $P = \Delta^n$. 
\end{proposition}
\begin{proof}, 
From Proposition~\ref{prop:uniformprojective}, we may assume $S_1$ is the simple such that $P/Pt\simeq S_1^{\oplus r}$. The composition factors of any finite length quotient of $P$ all lie in the $t$-orbit of $S_1$. Let $Q$ be any finitely generated projective module. If the $M_n(K\Delta)$-module $Q \otimes_R K \simeq (K\Delta^n)^{\oplus a}$, then by clearing denominators, we can find an embedding $Q \hookrightarrow P^{\oplus a}$ such that cokernel of $Qt \hookrightarrow P^{\oplus a}$ has finite length. It follows that $Q/Qt$ has composition factors in the $t$-orbit of $S_1$ and we are done.  
\end{proof}

For our structure theory, we will need the following order
$$
\Delta_d := 
\begin{pmatrix}
\Delta & \Delta & \cdots & \Delta \\
\textup{rad} \Delta & \Delta & & \vdots \\
\vdots & \ddots & \ddots & \vdots \\
\textup{rad} \Delta & \cdots & \textup{rad} \Delta & \Delta
\end{pmatrix}
\subseteq M_d(\Delta)
$$
If $\pi \in \Delta$ is a generator for $\text{rad}\, \Delta$, then one readily shows that 
$$
\begin{pmatrix}
0 & 1 & 0 & \cdots & 0 \\
\vdots & 0 & \ddots & \ddots& \vdots \\
\vdots & & \ddots & &  0\\
0 &  &  & & 1 \\
\pi  & 0 & \cdots & & 0
\end{pmatrix} \in \Delta_d.
$$
generates the radical on the left and right so $\Delta_d$ is normal. 

The order $\Delta_d$ arises naturally as per the following
\begin{proposition}  \label{prop:computehomsofPtj}
Let $P$ be the projective $\Lambda$-module $\Delta^n$.
\begin{enumerate}
    \item As subsets of $\Hom_{M_n(K\Delta)}(P \otimes_R K, P \otimes_R K) = K\Delta$, we have
    $$
    \Hom_{\Lambda}(Pt^i,Pt^j) = 
    \begin{cases}
    \Delta & \text{if } 0 \leq i-j < d \\
    \textup{rad}\, \Delta & \text{if } 0 < j-i \leq d
    \end{cases}
    $$
    \item In particular, $\End_{\Lambda} (P \oplus Pt \oplus \ldots \oplus Pt^{d-1}) = \Delta_d$. 
\end{enumerate}
\end{proposition}
\begin{proof}
Part (2) follows from (1) which we now prove. We may as well assume that $i = 0$. Note firstly that $\Hom_{\Lambda}(P,Pt)$ is the kernel of the natural surjection $\End_{\Lambda} P \to \End_{\Lambda} P/Pt$ which in turn is $\text{rad}\, \Delta$. For $-d< j \leq 0$, the composition factors of $Pt^j/P$ are all non-isomorphic to the simple summands of $P/Pt$. Hence $\Hom_{\Lambda}(P,Pt^j) = \Hom_{\Lambda}(P,P) = \Delta$ in this case. A similar argument shows that 
$$
\Hom_{\Lambda}(P,Pt) = \Hom_{\Lambda}(P,Pt^2) = \ldots = \Hom_{\Lambda}(P,Pt^d)
.$$
\end{proof}

\begin{theorem}  \label{thm:normaluptoMorita}
There exists $b | r$ such that $M_b(\Lambda) \simeq M_c(\Delta_d)$.
\end{theorem}
\begin{proof}
We know from Proposition~\ref{prop:cyclesimples}, that conjugation by $t$ permutes the Wedderburn components of $\Lambda/\text{rad}\, \Lambda$ so $\Lambda_{\Lambda}$ must be the projective cover of a semisimple module of the form $(S_1 \oplus \ldots \oplus S_d)^{\oplus a}$ for some $a$. We can find some $b|r$ such that $c = ba/r$ is an integer. It follows that $\Lambda^{\oplus b} \simeq (P \oplus Pt \oplus \ldots \oplus Pt^{d-1})^{\oplus c}$. Using Proposition~\ref{prop:computehomsofPtj} to compute endomorphism rings of both sides give the theorem. 
\end{proof}

\begin{corollary}  \label{cor:uniquenormal}
Up to isomorphism, a normal $R$-order $\Lambda$ is uniquely determined by the Brauer class of $K\Delta$ (in $\text{Br}\, K$) the number of simples of $\Lambda$ and the degree (or $R$-rank) of $\Lambda$. 
\end{corollary}
\begin{proof}
Suppose the Morita equivalence between $\Lambda$ and $\Delta_d$ is given by the Morita bimodule $ \ _{\Lambda}Q_{\Delta_d}$. If $S$ is the direct sum of one copy of each simple $\Delta_d$-module, then $Q_{\Delta_d}$  is the projective cover of a semisimple module $S^{\oplus a}$ for some $a$. The degree of $\Lambda$ determines $a$ uniquely. 
\end{proof}

To determine $\Lambda$ itself, it suffices to classify all possible indecomposable projective $\Delta_d$-module which we do now. 

Let $\bar{\Delta} = \Delta / \text{rad}\, \Delta$. We define a {\em $\bar{\Delta}$-flag} to be a sequence of $\bar{\Delta}$-submodules
$$
0 \leq \bar{I}_1 \leq \bar{I}_2 \leq \ldots \leq \bar{I}_d = \bar{\Delta}.
$$
Their inverse images in $\Delta$ gives the sequence of $\Delta$-modules
$$
\text{rad}\, \Delta \leq I_1 \leq I_2 \leq \ldots \leq I_d = \bar{\Delta}.
$$
Note that the module of row vectors $Q = (I_1 \ I_2 \ \ldots \ I_d)$ defines a $\Delta_d$-submodule of $\Delta^d$. It is projective since $\Delta_d$ is normal. 
\begin{proposition}  \label{prop:classifyindecprojovernormal}
The order $\Delta_d$ is normal.
\begin{enumerate}
\item The projective $\Delta_d$-module $Q$ constructed from a $\bar{\Delta}$-flag as above is indecomposable.
\item Every indecomposable projective $\Delta_d$-module has this form.
\item In particular, the indecomposable projectives are precisely the projective covers of any direct sum of $r$ simple $\Delta_d$-modules.
\end{enumerate}
\end{proposition}
\begin{proof}
It is easy to see that $\Delta_d$ is normal by constructing a left and right generator $t$ for $\text{rad}\, \Delta_d$. Note that $Q\otimes_R K \simeq (K\Delta)^d$ is an indecomposable $M_d(K\Delta)$-module so $Q$ is also indecomposable. 

If $S$ denotes a simple $\Delta$-module, then the simple $\Delta_d$-modules are 
$$
(S \ 0 \ldots 0), \quad (0 \ S \ 0 \ldots 0), \ldots, \quad ( 0 \ldots 0 \ S).
$$
Thus $Q$ is the projective cover of the semisimple module
$$
Q/Q (\text{rad}\, \Delta_d) \simeq (\bar{I}_1 \ I_2/I_1 \ldots I_d/I_{d-1})
$$
which is a direct sum of exactly $r$ simples. By varying the $\bar{\Delta}$-flag, we can construct the projective cover of any direct sum of $r$ simples we like. It thus remains to show there are no other indecomposable projective modules. Let $L$ be one such and suppose $L/L(\text{rad}\, \Delta_d)$ is a direct sum of more than $r$ simples. Then we can find a direct summand $T$ consisting of precisely $r$ simples and use an appropriate $\bar{\Delta}$-flag to construct the projective cover $Q$ of $T$. Then the natural surjection $L \to T$ lifts to a surjection $L \to Q$ which must split, a contradiction. If on the other hand, $L/L(\text{rad}\, \Delta_d)$ had fewer than $r$ simples, then we could apply the same argument to show that some projective $Q$ constructed using a $\bar{\Delta}$-flag decomposes, another contradiction.
\end{proof}

\bibliographystyle{amsalpha}

\bibliography{references}

\end{document}